\documentclass[12pt,a4paper]{amsart}
\usepackage{amsfonts}
\usepackage{amsthm}
\usepackage{amsmath}
\usepackage{amscd}
\usepackage{bbm}
\usepackage[latin2]{inputenc}
\usepackage{t1enc}
\usepackage[mathscr]{eucal}
\usepackage{indentfirst}
\usepackage{graphicx}
\usepackage{graphics}
\usepackage{pict2e}
\usepackage{epic}
\numberwithin{equation}{section}
\usepackage[margin=2.9cm]{geometry}
\usepackage{epstopdf}

\theoremstyle{plain}
\newtheorem{Th}{Theorem}[section]
\newtheorem{Lemma}[Th]{Lemma}
\newtheorem{Cor}[Th]{Corollary}
\newtheorem{Prop}[Th]{Proposition}

 \theoremstyle{definition}
\newtheorem{Def}[Th]{Definition}
\newtheorem{Conj}[Th]{Conjecture}
\newtheorem{Rem}[Th]{Remark}
\newtheorem{?}[Th]{Problem}
\newtheorem{Ex}[Th]{Example}

\newcommand*{\defeq}{\mathrel{\vcenter{\baselineskip0.5ex \lineskiplimit0pt
                     \hbox{\scriptsize.}\hbox{\scriptsize.}}}%
                     =}
                     
\newcommand{\forceindent}{\leavevmode{\parindent=2em\indent}}

\begin{document}

\title{The Number of $k$-Cycles in a Family of Restricted Permutations}

\author[E. Ozel]{Enes Ozel}

\address{University of Southern California \\ Department of Mathematics \\
3620 S. Vermont Avenue \\ Kaprielian Hall 104 \\ Los Angeles, CA} %\&  E\"{o}tv\"{o}s Lor\'{a}nd University \\ Department of Computer 
%Science \\ H-1117 Budapest
%\\ P\'{a}zm\'{a}ny P\'{e}ter s\'{e}t\'{a}ny 1/C \\ Hungary} 

\email{enozel@gmail.com}

% \subjclass[2010]{Primary: 05A05. Secondary: 05A19}

\subjclass[2010]{Primary: 05A05. Secondary: 05A19, 60C05}

 \keywords{permutations with restricted positions, cycle index, central limit theorem, compositions}

\begin{abstract}
In this paper we study different restrictions imposed over the set of permutations of size $n$, $S_n$, and for specific classes of restrictions study the cycle structure of the corresponding permutations. Specifically, we prove that for any fixed positive integer $k$, the number of $k$-cycles of a randomly chosen permutation $\pi\in S_n$ with the restriction ``$\pi(i) \geq i-1$ for $i\in\{2,...,n\}$'' has a Normal asymptotic distribution. We further prove that this result translates into CLTs regarding multiplicities of fixed-size parts of a uniformly selected composition of $n$.
\end{abstract}

\maketitle

%\vspace{-0.02in}
\section{Introduction}

\forceindent Permutations with restricted positions are studied under many different contexts. Hanlon \cite{Hanlon} studies Markov chains defined on permutations that adhere to one-sided restrictions. Klove \cite{Klove} and Baltic \cite{Baltic} work on permutations with limited displacement and provide the generating function for number of permutations in some cases. This type of restrictions are used in error-correcting codes. Efron and Petrosian \cite{EfronPetrosian1} use independence tests based on restricted permutations to analyze truncated astronomical data. Many more examples can be found in the paper by Diaconis, Graham and Holmes \cite{DiaconisGrahamHolmes}. They provide a wide survey of permutations with restricted positions and applications. 

\forceindent One way to answer the question ``How does a typical restricted permutation look like?'' is through understanding the distribution of its cycles, given that we chose the restricted permutation uniformly random. In this paper we study the cycle structures of specific families of restricted permutations and provide the asymptotic distribution for the number of cycles of fixed sizes for these permutations. Throughout this paper, we let $C_{n,k}(\pi)$ to be the number of $k$-cycles of the permutation $\pi\in S_n$. We will investigate the random variable $C_{n,k}(\pi)$, where we assume $\pi$ is a random permutation, chosen uniformly from the set of permutations adhering to the imposed restriction.

\forceindent The most general way to define restrictions to be imposed upon the symmetric group, $S_n$, is by a zero-one matrix. Let $M$ be an $n\times n$ matrix such that $M_{ij} = 1$ if the assignment $\pi (i) = j$ is allowed, and $M_{ij} = 0$ otherwise. For example, a matrix with ones everywhere except the main diagonal would result in the set of derangements, i.e., permutations with no fixed points.

\forceindent We define $S_M$ to be the set of permutations in accordance with the restriction matrix $M$, namely

\[
S_M = \{ \pi\in S_n : \prod_{i=1}^n M_{i \pi(i)} = 1 \}.
\]

Number of permutations in $S_M$ is given by the \emph{permanent} of $M$, 

\[
|S_M| = Per(M) = \sum_{\pi\in S_n} \prod_{i=1}^n M_{i\pi(i)}.
\]

\forceindent Permanent of a matrix is similar to the determinant, without the negative signs at every other iteration. Calculation of the permanent of a zero-one matrix is a widely studied problem, which was introduced by Valiant as the first example of a $\#$-P complete problem in \cite{Valiant}. The fastest algorithm, of $\mathcal{O}(n 2^n)$, that calculates the permanent of a zero-one matrix exactly is due to Ryser. Jerrum, Sinclair and Vigoda \cite{JerrumSinclairVigoda} provide a fully polynomial randomized approximation scheme to compute the permanent to any degree of approximation. For a survey of permanents see Diaconis et al \cite{DiaconisGrahamHolmes}. Permanents allow many calculations other than the number of restricted permutations, such as different moments of number of cycles of a fixed size.

% Hanlon's regular permutations and vector restrictions

\forceindent A particular case of restrictions can be described as ``one-sided''. One-sided restrictions are in the form of lower-bounds, i.e., $\pi(i) > j$. Such restrictions correspond to matrices with having zeroes only at their lower-diagonal part. Hanlon \cite{Hanlon} provides a brief survey of permutations with one-sided restrictions and studies the random transposition walk on these permutations.

\forceindent Hanlon calls the permutations with one-sided restrictions as ``$b$-regular permutations'', where $b$ is a vector of $n$ non-decreasing, positive integers. Then, restrictions are expressed as $\pi(i) \geq b_i$ for all $i$. One immediate result is that $b_i \leq i$ in order for the restriction to be nontrivial. Thinking in terms of the corresponding restriction matrix, the entries of ones on the $i$th row start at column $b_i$.We will denote the set of $b$-regular permutations as $S_b$. Here the dimension will be denoted implicitly through the vector $b$. The number of such permutations, as well as the generating functions for the number of cycles and number of inversions are well-known and are provided in \cite{Hanlon}. We will make use of the following formula for the number of $b$-regular permutations 

\begin{equation}\label{restvecformula}
|S_b| = \prod_{i=1}^n (1 + i - b_i)
\end{equation}
 
\noindent to show the following moment calculations on the number fixed points.

% Fixed point and transposition expected values and variances

%\begin{Prop}\label{meanformula}
%Let $\pi$ be a randomly chosen $b$-regular permutation. Then 
%$$\mathbb{E}[C_{n,1}(\pi)] = \frac{1}{|S_b|} \,\,\, \underset{k=1}{\overset{n}{\sum}} |S_{b,k}|$$
%\noindent and
%$$\mathbb{V}ar(C_{n,1}(\pi)) = \frac{1}{|S_b|} \,\,\, \sum_{k=1}^n \Big[ S_{b,k} (S_b - S_{b,k}) \Big] + 2  \sum_{i < j} \Big[ S_{b,i,j} - S_{b,i} S_{b,j} \Big],$$
%where $\left|S_{b,k}\right| = \underset{j=1}{\overset{k-1}{\prod}} \left(1+j-b_j\right) \times \underset{j=k+1}{\overset{n}{\prod}} \left(j-\left(b_j - \mathbbm{1}(b_j > k)\right)\right)$ and 
%\begin{equation*}%\label{Sbij}
%  \begin{alignedat}{2}
%    \Big| S_{b,i,j} \Big| = \prod\limits_{m=1}^{i-1} \Big(1+m-b_m \Big) &\times \prod\limits_{m=i+1}^{j-1} \Big(1+m- (b_{m+1} - \mathbbm{1}(b_{m+1}>i))  \Big) \\
%    & \times \prod\limits_{m=j+1}^{n} \Big(1+m- (b_{m+2} - \mathbbm{1}(b_{m+2}>i)) - \mathbbm{1}(b_{m+2}>j)) \Big).
%  \end{alignedat}
%\end{equation*}
%\end{Prop}

\forceindent We are interested in certain restriction vectors, namely when the restriction is $\pi(i) > i-1$. We study the cycle structure for these permutations, specifically, we find the local dependence relations of the cycle indicators. Using a local dependence result based on Stein's method, we prove central limit theorems for the number of $k$-cycles, for any $k$. The following is our main result.

%\vspace{0.2in}

\begin{Th}\label{r=2clt} %bunu Wasserstein metrikle yaz, Kolmorov'a selam cak
Consider the set of permutations $\pi$ such that $\pi(i) \geq i-1$ for $i\in\{2,...,n\}$. The corresponding restriction vector is $b_2 = [1, 1, 2, 3, \ldots, n-1]$. For any positive integer $k\leq n$, the number of $k$-cycles of a randomly selected permutation adhering to the above restriction, $C_{n,k}$ has a Normal limiting distribution. Specifically, as $n\to\infty$

$$
P\left( \frac{C_{n,k} - \mu_{n,k}}{\sigma_{n,k}} < t \right) \to \frac{1}{\sqrt{2\pi}} \int_{-\infty}^t \! e^{-x^2/2} \, \mathrm{d}x,
$$
\end{Th}

where

$$
\mu_{n,k} = \frac{n-k+3}{2^{k+1}} \,\,\,\, \mbox{ and } \,\,\,\, \sigma^2_{n,k} = \frac{(2^{k+1}-2k+3)n + 3k(k-4) + (3-k)2^{k+1}+ 5}{4^{k+1}}.
$$

%\vspace{0.2in}

\forceindent We will actually provide an upper bound to the Wasserstein distance between distribution functions of the above standardized version of $C_{n,k}$ and standard Normal distribution, which diminishes as $n$ increases. Theorem $\ref{r=2clt}$ follows as an immediate corollary.

%more about compositions? 

\forceindent We present a useful bijection that can be set up between these permutations and compositions of $n$. An ordered list of positive integers is called a composition of $n$, if the numbers in the list add up to $n$. In simple terms, compositions are partitions for which the order of parts matter. We use this bijection to calculate the cycle index of $b_2$-regular permutations and certain properties of compositions, which to our knowledge are proven for the first time.

%\newpage

\section{Restricted Permutations and Number of $k$-Cycles}

\forceindent In this section we investigate vector restrictions in detail and step by step prove Theorem \ref{r=2clt}. Specifically, in Section \ref{permcalc} we use permanents to calculate the expected value and variance of number of fixed points for any vector restricted permutation.  In Section \ref{bijcyc} we begin the discussion of our main subject, $b_2$-regular permutations introduced above. Here we start our investigation about how the cycles of these permutations look like, then we present the bijection between them and the compositions of $n$. After deriving the cycle index, using it we calculate the moments in the central limit theorem. We are indebted to Prof. Jason Fulman\footnote{University of Southern California, Department of Mathematics} for the idea of using the cycle index and related derivations. We study the local dependence structure for the indicator random variables, whose sum gives the subject random variables $C_{n,k}$, in Section \ref{localdep}. We then proceed to combining these ingredients, describing the Stein's method set up we use and proving Theorem \ref{r=2clt} in Section \ref{steinclt}. We discuss how these results translate to compositions. Lastly, we provide ways to generalize the results presented in this work, potential results and differences with the case we proved.

\subsection{Permanent Calculations}\label{permcalc}% we use permanents to calculate the expected value and variance of number of fixed points for any vector restricted permutation

\forceindent Finding the asymptotic distribution of the number of $k$-cycles for a randomly selected permutation in $S_n$ is a classical problem, for which Feller \cite{Feller} is the most accessible reference. The number of $k$-cycles converges to a Poisson random variable with mean $1/k$. Another great reference is Arratia and Tavare \cite{ArratiaTavare}; they provide an upper bound to the total variation distance between the distributions of $(C_{n,1}, C_{n,2},\ldots )$ and $(L_1, L_2,\ldots )$, where $L_i$ are independent Poisson random variables mentioned above, with mean $1/i$.

%the general problem of identifying which cycle is allowed
\forceindent Differently than the classical unrestricted case, we first need to identify which cycles are allowed and which are not due to the restrictions. Depending on the restrictions the distributions of the number of $k$-cycles might change critically. For example, in the case of $M$ being the matrix with $1$'s everywhere except the main subdiagonal, there could be no fixed points, where without the restriction there would be around one fixed point on average. Whatever the imposed restriction might be, we cannot ignore these changes. One way to study the problem is through expressing the number of $k$-cycles as a sum of indicators, i.e., 

\[
C_{n,k}(\pi) = \sum_{\mbox{c is an allowed k-cycle}} \mathbbm{1}\{\mbox{c is a cycle of } \pi\}.
\]

%matrix conditionings, remark on conditionings resulting in similar but smaller vectors, an iterative approach

\forceindent For a given vector $b$ we calculate the expected value and variance of the number of fixed points $C_{n,1}(\pi)$, where $\pi$ is a randomly chosen $b$-regular permutation. Using the indicator representation of the number of fixed points, $C_{n,1}(\pi) = \sum_{i=1}^n \mathbbm{1}\{\pi(i) = i\}$, we need to calculate probabilities of the form $P(\pi(i) = i)$, which is equivalent to counting number of $b$-regular permutations with ``$\pi(i) = i$'' happening. We calculate the permanents of the matrices with the corresponding conditioning. For instance, consider the restriction vector $b = [1, 1, 2, 4, 4]$. The corresponding restriction matrix is

\[
M = \begin{bmatrix}
    1 & 1 & 1 & 1 & 1 \\
    1 & 1 & 1 & 1 & 1 \\
    0 & 1 & 1 & 1 & 1 \\
    0 & 0 & 0 & 1 & 1 \\
    0 & 0 & 0 & 1 & 1 \\
\end{bmatrix}.
\]

\forceindent If we were to condition on the event that $\{\pi(2)=2\}$, then as we prevent $\pi(2)$ to have any value other than $2$ and vice versa, the corresponding matrix would have a $1$ only at the position $(2,2)$ along the column and row number two:

\[
M^{(2)} = \begin{bmatrix}
    1 & 0 & 1 & 1 & 1 \\
    0 & 1 & 0 & 0 & 0 \\
    0 & 0 & 1 & 1 & 1 \\
    0 & 0 & 0 & 1 & 1 \\
    0 & 0 & 0 & 1 & 1 \\
\end{bmatrix}.
\]

\forceindent Permanent of $M^{(2)}$ is equal to the permanent of the smaller matrix where the row and column two are deleted:

\[
M^{(2)^*} = \begin{bmatrix}
    1 & 1 & 1 & 1 \\
    0 & 1 & 1 & 1 \\
    0 & 0 & 1 & 1 \\
    0 & 0 & 1 & 1 \\
\end{bmatrix}.
\]

\forceindent This new matrix corresponds to a new and smaller restriction vector. We can imagine this vector as a restriction to be imposed on the permutations of the four numbers $\{1, 3, 4, 5\}$. Permanent of this smaller matrix is equal to the formula (\ref{restvecformula}) applied to the new restriction vector, $b^{(2)} = [1, 2, 3, 3]$. To understand how the conditioned event ``$\pi(2) = 2$'' determines the corresponding vector $b^{(2)}$, we observe the fact that deleting the column and row $2$ would decrease the number of ones any row that would have $b_i \leq 2$. Otherwise they are not affected. So we have the following remark.

\begin{Rem}\label{vectorreduction}
The number of $b$-regular permutations with $\pi(i) = i$ equals to the number of $b^{(i)}$-regular permutations where $b^{(i)}$ is the vector obtained by erasing $i$th entry of $b$ and decreasing the entries $j > i$ given that $b_j^{(i)} > i$, otherwise leaving them as they are.
\end{Rem}

\forceindent This remark immediately leads to the calculation of the expected number of fixed points and its variance. 

%fixed point and transposition results

\begin{Prop}\label{meanformula}
Let $\pi$ be a randomly chosen $b$-regular permutation. Then 
$$\mathbb{E}[C_{n,1}(\pi)] = \frac{1}{|S_b|} \,\,\, \underset{k=1}{\overset{n}{\sum}} |S_{b,k}|$$
\noindent and
$$\mathbb{V}ar(C_{n,1}(\pi)) = \frac{1}{|S_b|} \,\,\, \sum_{k=1}^n \Big[ S_{b,k} (S_b - S_{b,k}) \Big] + 2  \sum_{i < j} \Big[ S_{b,i,j} - S_{b,i} S_{b,j} \Big],$$
where $\left|S_{b,k}\right| = \underset{j=1}{\overset{k-1}{\prod}} \left(1+j-b_j\right) \times \underset{j=k+1}{\overset{n}{\prod}} \left(j-\left(b_j - \mathbbm{1}(b_j > k)\right)\right)$ and 
\begin{equation*}%\label{Sbij}
  \begin{alignedat}{2}
    \Big| S_{b,i,j} \Big| = \prod\limits_{m=1}^{i-1} \Big(1+m-b_m \Big) &\times \prod\limits_{m=i+1}^{j-1} \Big(1+m- (b_{m+1} - \mathbbm{1}(b_{m+1}>i))  \Big) \\
    & \times \prod\limits_{m=j+1}^{n} \Big(1+m- (b_{m+2} - \mathbbm{1}(b_{m+2}>i)) - \mathbbm{1}(b_{m+2}>j)) \Big).
  \end{alignedat}
\end{equation*}
\end{Prop}

\begin{proof}
$S_b$ was the number of $b$-regular permutations, given in the formula (\ref{restvecformula}). By Remark \ref{vectorreduction} the number of $b$-regular permutations where $k$ is a fixed point is $|S_{b,k}|$ and the number of $b$-regular permutations with $i$ and $j$ are fixed points is $|S_{b,i,j}|$, using the remark twice. The results follow from the linearity of expected value and covariance formula for the variance of sum of indicators.
\end{proof}

%general discussion about the inflexibility of the method in terms of cost of calculations

\forceindent This method is applicable to other cycle types as well, granted that we know which of the cycles of that type are allowed. For $b$-regular permutations, any coordinate might be a fixed point by the assumption $b_i \leq i$. But for transpositions and bigger cycles this identification is necessary. Furthermore, conditioning on a $k$-cycle will require $k$ indicators to be subtracted and thus calculating higher moments will become more and more expensive. We will sacrifice generality for utility and use cycle index for a specific family of $b$-regular permutations.

\subsection{Hypercube Permutations and Cycle Index}\label{bijcyc}%include composition results here

%Hypercube permutations, motivation and their cycles

\forceindent For the rest of this paper, we will be mainly interested in a special form of vector restrictions, namely, when each coordinate has to be mapped to at least one coordinate below itself. This restriction may be expressed as ``$\pi(i) > i-1$'', for $i= 2, 3, \ldots, n$, or, through the restriction vector $b_2 = [1, 1, 2, 3, \ldots, n-1]$. The corresponding restriction matrix (for $n=5$) is

\[
\begin{bmatrix}
    1 & 1 & 1 & 1 & 1 \\
    1 & 1 & 1 & 1 & 1 \\
    0 & 1 & 1 & 1 & 1 \\
    0 & 0 & 1 & 1 & 1 \\
    0 & 0 & 0 & 1 & 1 \\
\end{bmatrix}.
\]

\vspace{0.1in}

\forceindent This matrix has one nonzero lower subdiagonal. We would like to study this form of restriction, where there are a fixed number of nonzero lower subdiagonals. We believe the ideas for the one lower subdiagonal restriction can be generalized, with effort, to any fixed number of lower subdiagonals.

\forceindent Using formula (\ref{restvecformula}) we see that there are $2^{n-1}$ $b_2$-regular permutations. In \cite{Hanlon} Hanlon constructs a bijection between $b_2$-regular permutations and the vertices of the hypercube. This bijection preserves the random walk structure, in the sense that any two $b_2$-regular permutations that are one transposition away from each other are mapped to two hypercube vertices that are adjacent. Similarly, we construct a bijection from $b_2$-regular permutations to the compositions of $n$. This bijection has the property of matching not only permutations to compositions, but also cycles of any size in the permutation to parts of the same size in the composition. This allows us to construct the cycle index of these permutations using properties of compositions.

\forceindent Before looking at the cycle index construction, we would like to have a look at the cycles themselves. It turns out that cycles of $b_2$ regular permutations have a very specific form.

\begin{Rem}\label{cycleremark}
Let $\pi$ be a $b_2$- regular permutation, then any $k$-cycle of $\pi$ will be of the form $(m_k m_{k-1} m_{k-2} \ldots m_2 m_1)$, when written in the cycle notation. Here $m_i$ is the $i$th greatest number within the cycle. Further, $m_i = m_k + k - i$, or, equivalently, $m_i$ are increasing, consecutive numbers.
\end{Rem}

\forceindent The reason for this follows from the restriction $\pi(i) > i-1$, i.e., the fact that $\pi(i)$ has to be mapped to at least one coordinate below. Say $m_k$ is the largest entry within the $k$-cycle, then it has to be mapped to $m_k-1$ unless it is a fixed point, i.e., $k=1$. Similarly, $m_k-1$ has to be mapped to $m_k-2$, and following iteratively we reach down to $m_1$, which has to be mapped to $m_k$, which is allowed as there are no upper bound restrictions. This form gives the clue for the bijection, the largest entries so far will be breaking points.

%bijection and examples

\begin{Def}
Let $\pi\in S_n$, then $i\in\{1, 2, \ldots, n\}$ is called a \emph{record position} of $\pi$, if $\pi(i) > \pi(j)$ for all $j < i$. Further, $\pi(i)$ is called a \emph{record value}.
\end{Def}

By this definition, $i=1$ is always a record position and this is inherently equivalent to the fact that every permutation has at least one cycle. Now we present the bijection.

\begin{Lemma}\label{bijection}
There is a bijection between the set of $b_2$-regular permutations and the set of compositions of $n$, which matches the cycle sizes of permutations to part sizes compositions.
\end{Lemma}

\begin{proof}
Let $\pi\in S_{b_2}$, then we consider the record positions of $\pi$. Starting at each record position, we proceed till the next record position to construct a part of of the composition. To give an example, if the record positions for $\pi\in S_{10}$ are $1, 2, 5, 6$, then the corresponding composition is $(1,3,1,5)$. We will prove that this mapping is one-to-one and as the number of compositions of $n$ is $2^{n-1}$ as well, the mapping will indeed be a bijection.

By Remark \ref{cycleremark} we can deduce that any record position will start a new cycle that will proceed till the next record position. This is as a record position is the greatest value till the next record position, thus the cycle it belongs to will start with it, and it will end with the next record position and a new cycle will begin. 

Lets assume we have two permutations $\pi_1, \pi_2 \in S_{b_2}$, and further, they are mapped to the same composition of $n$. By the definition of our mapping, these permutations have the same record positions and so the same cycle structure. The record values, too, must be the same, otherwise the permutation with the higher record value will cease one of its latter record positions from being a record position. $\pi_1$ and $\pi_2$ have cycles of identical length, with identical maximum value, so by Remark \ref{cycleremark} they have the same cycles, or, equivalently, $\pi_1 = \pi_2$. This proves that this mapping is one to one and therefore a bijection. Lastly, by the construction and Remark \ref{cycleremark} $k$-cycle always corresponds to a $k$-part as desired.

\end{proof}

%cycle index theorem and calculations of the first few moments

\forceindent We will use the following Lemma in the calculations for cycle index.

\begin{Lemma}\label{lemmo}
\[
2 + \sum_{n=0}^{\infty} \frac{u^{n+1}}{2^n} \sum_{\pi\in S_{b_2}} \frac{\prod_i x_i^{n_i(\pi)}}{\left(\sum n_i(\pi)\right)!} = 2 \prod_{i=1}^{\infty} e^{x_i\left(\frac{u}{2}\right)^i},
\]

where $n_i(\pi)$ is the number of $i$-cycles of $\pi$.
\end{Lemma}

\begin{proof}

By Lemma \ref{bijection} we know cycles of $\pi\in S_{b_2}$ corresponds to parts of the mapped composition. This implies the number of $b_2$-regular permutations with $n_i$ $i$-cycles is $\frac{\left(\sum_i n_i\right)!}{\prod_i n_i!}$, under the assumption that $\sum_i i n_i = n$. Therefore

\begin{align*}
1 + \sum_{n=0}^{\infty} u^{n+1} \sum_{\pi\in S_{b_2}} \frac{\prod_i x_i^{n_i(\pi)}}{\left(\sum_i n_i(\pi)\right)!} &= \prod_{i=1}^{\infty} \sum_{n_i=0}^{\infty} \frac{(x_iu^i)^{n_i}}{n_i!} \\
&= \prod_{i=1}^{\infty} e^{x_iu^i}
\end{align*}

\noindent Replacing $u$ by $u/2$ and multiplying through by $2$ gives the desired equation.

\end{proof}

\forceindent Now we derive the cycle index for $b_2$ regular permutations.

\begin{Th}\label{cycleindex}

The cycle index for $S_{b_2}$ is given by

\[
\sum_{n=0}^{\infty} \frac{u^{n+1}}{2^n} \sum_{\pi\in S_{b_2}} \prod_i x_i^{n_i(\pi)} = 2 \sum_{k\geq 1} \left(\sum_{i\geq 1} x_i \left(\frac{u}{2}\right)^i \right)^k.
\]

\end{Th}

\begin{proof}

In the Lemma \ref{lemmo} we replace each $x_i$ with $t x_i$, we obtain

\[
2 + \sum_{n=0}^{\infty} \frac{u^{n+1}}{2^n} \sum_{\substack{\pi\in S_{b_2} \\ \sum n_i(\pi) = k}} t^k \frac{\prod_i x_i^{n_i(\pi)}}{k!} = 2 \prod_{i=1}^{\infty} e^{t x_i\left(\frac{u}{2}\right)^i}
\]

Differentiating $k$ times with respect to $t$ and then setting $t=0$ results

\[
\sum_{n=0}^{\infty} \frac{u^{n+1}}{2^n} \sum_{\substack{\pi\in S_{b_2} \\ \sum n_i(\pi) = k}} \prod_i x_i^{n_i(\pi)} = 2 \left(\sum_{i\geq1} x_i \left(\frac{u}{2}\right)^i \right)^k.
\]

Summing this equation over all $k$ we obtain the desired equation, and our cycle index.

\end{proof}

We will now present an example calculation and compare it to Proposition 2.2.%\ref{meanformula}.

\begin{Ex}\label{twomomentcalc}
We calculate the expected number of fixed points for a random $b_2$-regular permutation using both methods. 

(i) We apply Proposition 2.2 to the vector $b_2 = [1, 1, 2, 3, \ldots, n-1]$ and get

\begin{align*}
\mathbb{E}[C_{n,1}(\pi)] &= \frac{1}{2^{n-1}} \left(2^{n-2} + 2^{n-3} + \ldots + 2^{n-3} + 2^{n-2}\right) \\
&= \frac{1}{2^{n-1}} \left(2 2^{n-2} + (n-2) 2^{n-3}\right) \\
&= \frac{4 2^{n-3} + (n-2) 2^{n-3}}{2^{n-1}} \\
&= \frac{4 + (n-2)}{2^2} \\
&= \frac{n+2}{4}.
\end{align*}

(ii) We use the cycle index in Theorem \ref{cycleindex}. By setting $x_1 = x$ and $x_i = 1$ for $i>1$, we obtain the gerenating function for the number of fixed points,

\[
2 \sum_{k\geq 1} \left(\frac{ux}{2} + \sum_{i\geq 2} \left(\frac{u}{2}\right)^i\right)^k.
\]

Using formal power series, this can be rewritten as

\[
2 \left(\frac{2(2-u)}{(2-u)(2-ux)-u^2} - 1\right).
\]

Taking its derivative with respect to $x$ and setting $x=0$ results in the generating function for the first moments,

 \[
 \frac{4u(2-u)^2}{\left(4(1-u)\right)^2}.
 \] 
 
Expanding this power series, we get

\begin{align*}
 \frac{4u(2-u)^2}{\left(4(1-u)\right)^2} &= \frac{u}{4} \left(\frac{2-u}{1-u}\right)^2 \\
 &= \frac{u}{4} \left(1 + \frac{1}{1-u}\right)^2\\
 &= \frac{u}{4} \left(1 + 2 \sum_{n\geq 0} u^n + \sum_{n\geq 0} (1+n)u^n\right)\\
 &= \frac{u}{4} \left(1 + \sum_{n\geq 0} (n+3) u^n\right)\\
 &= \frac{u}{4} + \sum_{n\geq 1} \frac{n+2}{4} u^n.
\end{align*}

The general term of the power series, $\frac{n+2}{4}$, is the desired quantity.

\end{Ex}

Similarly, we can calculate any moment of $k$-cycles, by setting $x_k = x$ and other $x_i = 1$ and differentiating as many times as necessary to obtain the desired degree of moment.

\begin{Th}\label{twomomentsofkcycles}

Let $C_{n,k}$ be the number of $k$-cycles of a uniformly selected $b_2$-regular permutation. Then 

\[
\mu_{n,k} \defeq \mathbb{E}[C_{n,k}] = \frac{n-k+3}{2^{k+1}},
\]

\noindent and 

\[
\mathbb{E}[C_{n,k}(C_{n,k}-1)] = \frac{(n+2-2k)(n+7-2k)}{4^{k+1}},
\]

\noindent combining we get

\[
\sigma^2_{n,k} \defeq Var(C_{n,k}) = \frac{(2^{k+1}-2k+3)n + 3k(k-4) + (3-k)2^{k+1}+ 5}{4^{k+1}}.
\]

\end{Th}

\begin{proof}

Proceeding as in Example \ref{twomomentcalc}, we set $x_k = x$ and $x_i = 1$. Differentiating once gives the expected value, differentiating twice gives the second falling moment. Adding the correct amount of the first moment gives the variance.

\end{proof}

These numbers appear in the central limit theorem, Theorem \ref{r=2clt}. Before progressing further into the main result, we state a result on compositions of $n$. Using the first moments we can easily prove, for example, that number of $1$'s that occur in compositions of $5$ is the same as the number of $2$'s that occur in compositions of $6$.

% results about composition, specifically the one about the number of 1's in 10 vs 2's in 11.

\begin{Prop}
The total number of $k$-parts in compositions of $n$ is the same as the total number of $k+m$-parts in compositions of $n+m$. 
\end{Prop}

\begin{proof}

By Proposition \ref{twomomentsofkcycles}, we have the total number of $k$-parts in compositions of $n$

\[
total \# k-parts = \frac{n-k+3}{2^{k+1}} 2^{n-1} = (n-k+3) 2^{n-k-2}.
\]

In this formula replacing $k$ with $k+m$ and $n$ with $n+m$ for any positive integer $m$ results in the same number, completing the proof.

\end{proof}

Working with different moments, many other results regarding the behavior of $k$-parts of a random composition can be proven.

In the next section we will be investigating how occurance of each cycle affects the other cycles. This will be crucial in the proof of the central limit theorem.

%HERE AND THERE
\subsection{Dependence Structure of Indicators}\label{localdep}

%Which cycles are allowed? Construct them. 

\forceindent As stated in Section 2.1, we consider the number of $k$-cycles as sums of indicators. By Remark \ref{cycleremark} we can identify all allowed cycles in a $b_2$-regular permutation. Using the fact that $\pi(i) \geq i-1$, we may set $\pi(i) = i$, i.e., any coordinate might be a fixed point. So we get $C_{n,1}(\pi) = \sum_{i=1}^n \mathbbm{1}\{\pi(i) = i\}$. Similarly, only the adjacent transposition are allowed, as otherwise we would violate Remark \ref{cycleremark} by mapping some coordinate $i$ to $\pi(i) < i-1$. Thus, we get $C_{n,2}(\pi) = \sum_{i=1}^{n-1} \mathbbm{1}\{\pi(i) = i+1, \pi(i+1) = i\}$.

\forceindent It will proceed similarly for other cycle sizes. There are $n-k+1$ potential $k$-cycles, as the ``first'' of them has $k$ as its largest element and the ``last'' of them has $n$ as its largest element. We will denote the indicators of these $k$ cycles as ``$I^k_i$'', where $i= 1, \ldots, n-k+1$. In this chapter, our objective is to identify the dependence relationship between $I^k_i$ and $I^k_j$. Specifically, we will prove that $I^k_i$ and $I^k_j$ are independent if and only if $|i-j|\geq k+1$. 

\forceindent Conditioning on the existence of a $k$-cycle will result in two potential sitautions, depending on if the $k$-cycle is positioned in one of the end points, being the ``first'' or the ``last'' $k$-cycle or not. We will explain this through the bijection we defined in Lemma \ref{bijection}. Lets first condition on the existence of $I^k_1$ or $I^k_{n-k+1}$. This conditioning leads us to compositions with a $k$-part in one of the endpoints, leaving $n-k$ remaning to be composed. There are $2^{n-k-1}$ compositions of $n-k$, therefore the probability that $I^k_1$ (or,  $I^k_{n-k+1}$) occurs is 

\[
P(I^k_1 occurs) = P(I^k_{n-k+1} occurs) = \frac{2^{n-k-1}}{2^{n-1}} = \frac{1}{2^{k}}.
\]\vspace{0.1in}

If we want one of the ``mid'' $k$-cycles to occur, then this would separate the remaning parts into two pieces, say of sizes $a$ and $n-k-a$. Then, there are $2^{a-1} 2^{n-k-a-1} = 2^{n-k-2}$ ways to compose these two parts, resulting in

\[
P(I^k_{j} occurs) = \frac{2^{n-k-2}}{2^{n-1}} = \frac{1}{2^{k+1}}, \mbox{ for } j = 2, 3, \ldots, n-k.
\]\vspace{0.1in}

\forceindent Now lets go back to conditioning on the occurance of $I^k_1$. First of all, any other potential $k$-cycle that would use any number in $I^k_1$, namely $1, 2, \ldots, k$, would become impossible. As $I^k_1$ corresponds to a $k$-part, the resulting subspace will be identical to a composition of $n-k$. Here, within this remaining $n-k$ part the probability for any $k$-part stays the same, except for the cycle that starts immediately after $I^k_1$. This $k$-cycle, $I^k_{k}$ used to be a ``mid'' $k$-cycle, with probability of happening $1/2^{k+1}$, but now its one of the end $k$-cycles of the remaining $n-k$ composition, bumping its probability of happening to $1/2^k$. For $j= k+1, \ldots, n-k+1$, $I^k_j$ remain unaffected. Symmetrically, the occurance of the $k$-cycle $I^k_{n-k+1}$ prevents $k$-cycles $I^k_{n-k}, \ldots, I^k_{n-2k+1}$ would become impossible, and $I^k_{n-2k}$ would change its probability of happening from $1/2^{k+1}$ to $1/2^k$. The remaning $k$-cycles would remain unaffected.

\forceindent A similar argument can be done for $I^k_j$, for $j= 2, \ldots, (n-k)$, where any $k$-cycle that uses a value in $I^k_j$  becomes impossible. Any $k$-cycle with a value that is adjacent to a value in $I^k_j$ changes its probability either from $1/2^k$ to $1/2^{k+1}$ or vice versa. 

\begin{Def}
Let $c$ and $d$ be any two different cycles of a permutation $\pi$, where elements of $c$ are $c_1, c_2, \ldots, c_k$ and elements of $d$ are $d_1, d_2, \ldots, d_m$. Then the cycles $c$ and $d$ are called \emph{$p$-distant} if
\[
\min_{\substack{i\in\{1,\ldots,k\} \\ j\in\{1,\ldots,m\}}} |c_i - d_j| > p
\]

and $p$ is the largest number satisfying this inequality.
\end{Def}

Combining these observations, we reach to the following independence structure.

\begin{Prop}
Let $I^k_1, I^k_2, \dots, I^k_{n-k+1}$ be the potential $k$-cycles of a random $b_2$-regular permutation. Then the indicators for $I^k_i$ and $I^k_i$ are independent if and only if these two cycles are at least $1$-distant, i.e., $|i-j| > k+1$ in terms of index of cycles.
\end{Prop}

\begin{proof}
WLOG assuming $i < j$ and using the same reasoning before, for occurance of $I^k_i = (i\,\,\,\, i+k \,\,\,\, i+k-1 \,\,\,\, i+k-2 \ldots i+1)$, written using the cycle notation, not to affect occurance of $I^k_j = (j \,\,\,\, j+k \,\,\,\, j+k-1  \,\,\,\, j+k-2 \ldots j+1)$, we must have the maximum value $i+k$ in $I^k_i$ and the minimum value $j$ in $I^k_j$ to be at least $1$-distant. This happens when $j-i > k+1$. Thus, $|i-j| > k+1$ implies the independence of the indicators for $I^k_i$ and $I^k_j$ (and vice versa).
\end{proof}

\forceindent This form of dependence is sometimes referred to as ``$m$-dependence''. There are many central limit theorems for $m$-dependent random variables. A strong result (Theorem 4.2) that is applicable can be found in \cite{Chen}. 

\forceindent We will use a result of a much deeper and more comprehensive method, called ``Stein's Method'', as it will provide an easy calculation for the upper bound of the approximation error.

%\newpage

\subsection{Stein's Method and the CLT}\label{steinclt}

\forceindent Stein's Method provides means to establish approximations to many well-known probability distributions. This method has been studied extensively since Stein's seminal paper \cite{Stein} in 1972. It includes a vast arsenal of probabilistic tools that are applicable in many different problems. One special case is when we know the underlying local dependence structure and want to prove the sum of these dependent random variables have a Normal limiting distribution. A general description is presented in chapter 9 of the text ``Normal Approximation by Stein's Method'', by Chen, Goldstein and Shao \cite{ChenGoldsteinShao}. The result we actually use (Theorem 3.6) is proved in Ross's survey \cite{Ross} on Stein's method. We will state it without proving.

\begin{Th}[Local Dependence CLT]\label{localdepclt}
Let $X_1, \ldots, X_n$ be random variables such that $\mathbbm{E}[X_i^4] < \infty$, $\mathbbm{E}[X_i] = 0$, $\sigma^2 = Var(\sum_i X_i)$, and define $W = \sum_i X_i/\sigma$. Let the collection $(X_1, \ldots, X_n)$ have dependency neighborhoods $N_i$, $i= 1, \ldots, n$, and also define $D\defeq \max_{1\leq i\leq n} |N_i|$. Then for $Z$ a standard normal random variable,
\[
d_W(W,Z) \leq \frac{D^2}{\sigma^3} \sum_{i=1}^n \mathbbm{E}|X_i|^3 + \frac{\sqrt{28} D^{3/2}}{\sqrt{\pi} \sigma^2} \sqrt{\sum_{i=1}^n \mathbbm{E}[X_i^4]}.
\]\qed
\end{Th}

\forceindent Here $d_W$ is the \emph{Wasserstein} metric defined over probability measures. In general, many different metrics may be used to understand the distance between two probability measures, say $\mu$ and $\nu$, in the following general form:

\[
d_{\mathscr{H}}(\mu,\nu) = sup_{h\in\mathscr{H}}\left|\int h(x) \,d\mu(x) - \int h(x) \,d\nu(x)\right|,
\]

where $\mathscr{H}$ is a specific family of ``test'' functions. For the random variables $W$ and $Z$ (as in Theorem \ref{localdepclt}) with measures $\mu$ and $\nu$, we set $d_{\mathscr{H}}(W,Z) = d_{\mathscr{H}}(\mu,\nu)$, an abuse of notation.

\forceindent There are several choices for $\mathscr{H}$ used in different contexts. One of the most natural choices is when $\mathscr{H} = \{\mathbbm{1}[.\leq x]: x\in\mathbb{R}\}$, which leads to \emph{Kolmogorov} metric, denoted $d_K$. Kolmogorov metric is one of the most intuitive probability metrics, giving the maximum difference between the distribution functions of the random variables. When $\mathscr{H} = \{h:\mathbb{R} \to \mathbb{R} : |h(x) - h(y)| \leq |x - y|\}$, we obtain the aforementioned Wasserstein metric, denoted by $d_W$. Wasserstein metric is another commonly used measure of distance between two probability measures. We can use Wasserstein metric to bound Kolmogorov metric from above using the following inequality, stated as Proposition 1.2 in \cite{Ross}:

\[
d_K(W,Z) \leq \sqrt{2C d_W(W,Z)},
\]\vspace{0.1in}

where $C$ is an upper bound for the density function of $Z$. Now we are equipped to combine the moment calculations via the cycle index, local dependence structure of the indicators and local dependence central limit theorem, Theorem \ref{localdepclt}.

\forceindent For the indicators of $b_2$-regular permutation $k$-cycles, we observed that each cycle may have some intersection of elements with $k-1$ many $k$-cycles to its left and right, and it changed the probability of the indicators that were $1$-distant from the its endpoints. For a fixed indicator, these make up to a maximum of $(k-1) + 1 (k-1) + 1 = 2k$ dependent indicators. That means $D\defeq \max_{1\leq i\leq n} |N_i| = 2k$.

\forceindent To be able to utilize Theorem \ref{localdepclt}, we need to use random variables with mean zero. Recalling that the probability that for a $b_2$-regular permutation

\[
P(I_i^k occurs) = \begin{cases} 
      1/2^k & i= 1, n-k+1 \\
      1/2^{k+1} & i= 2, \ldots, n-k
   \end{cases}
\]

we let

\[
X_i = \begin{cases} 
      \mathbbm{1}(I_i^k occurs) - 1/2^k & i= 1, n-k+1 \\
      \mathbbm{1}(I_i^k occurs) - 1/2^{k+1} & i= 2, \ldots, n-k
   \end{cases}
\]

so that $\mathbbm{E}[X_i] = 0$. The values $\mu_{n,k}$ and $\sigma_{n,k}^2$ are still to be used. Variance of the sum of these shifted indicators is still $\sigma_{n,k}^2$ as the sum of these shifts still does not change the variance, and $\mu_{n,k}$ is this total shift amount. What we need to calculate are $\mathbbm{E}|X_i|^3$ and $\mathbbm{E}[X_i^4]$ for each $i$. By the definition of $X_i$, these values are different for $i= 1, n-k+1$ versus $i= 2,\ldots, n-k$.

For $i = 1, n-k+1$,

\[
|X_i| = \begin{cases} 
      1/2^k & \mbox{with probability} 1 - \frac{1}{2^k} \\
      1 - 1/2^{k} & \mbox{with probability} \frac{1}{2^k}
   \end{cases}
\]

and for $i = 2,\ldots, n-k$,

\[
|X_i| = \begin{cases} 
      1/2^{k+1} & \mbox{with probability} 1 - \frac{1}{2^{k+1}} \\
      1 - 1/2^{k+1} & \mbox{with probability} \frac{1}{2^{k+1}}
   \end{cases}.
\]

Calculating their third moments, for $i = 1, n-k+1$,

\begin{align*}
\mathbbm{E}|X_i|^3 &= \frac{1}{2^{3k}}\left(1-\frac{1}{2^k}\right) + \left(1 - \frac{1}{2^k}\right)^3 \frac{1}{2^k}\\
&= \frac{2^{3k} - 3.2^{2k} + 2^{k+2} - 2}{2^{4k}},
\end{align*}

and for $i = 2,\ldots, n-k$,

\begin{align*}
\mathbbm{E}|X_i|^3 &= \frac{1}{2^{3k+3}}\left(1-\frac{1}{2^{k+1}}\right) + \left(1 - \frac{1}{2^{k+1}}\right)^3 \frac{1}{2^{k+1}}\\
&= \frac{2^{3(k+1)} - 3.2^{2(k+1)} + 2^{k+3} - 2}{2^{4(k+1)}}.
\end{align*}

Summing these over $i$, we get

\begin{align*}
a_{n,k} &\defeq \sum_{i=1}^{n-k+1} \mathbbm{E}|X_i|^3 \\
 & = 2\left(\frac{2^{3k} - 3.2^{2k} + 2^{k+2} - 2}{2^{4k}}\right) + (n-k-1)\left(\frac{2^{3(k+1)} - 3.2^{2(k+1)} + 2^{k+3} - 2}{2^{4(k+1)}}\right)
\end{align*}

Now we calculate the fourth moments. For $i = 1, n-k+1$,

\begin{align*}
\mathbbm{E}[X_i^4] &= \left(\frac{-1}{2^k}\right)^4 \left(1 - \frac{1}{2^k}\right) + \left(1 - \frac{1}{2^k}\right)^4 \frac{1}{2^k} \\
&= \frac{(2^{2k} - 3 2^{k} + 3)(2^k - 1)}{2^{4k}},
\end{align*}

and for $i = 2,\ldots, n-k$,

\[
\mathbbm{E}[X_i^4] = \frac{(2^{2(k+1)} - 3 2^{k+1} + 3)(2^{k+1} - 1)}{2^{4(k+1)}}.
\]

Summing these over $i$, we get

\begin{align*}
b_{n,k} &\defeq \sum_{i=1}^{n-k+1} \mathbbm{E}[X_i^4] \\
& = 2\left(\frac{(2^{2k} - 3.2^{k} + 3)(2^k-1)}{2^{4k}}\right) + (n-k-1)\left(\frac{(2^{2(k+1)} - 3.2^{k+1} + 3)(2^{k+1} - 1)}{2^{4(k+1)}}\right).
\end{align*}

Now we combine all these ingredients within Theorem \ref{localdepclt}.

\begin{Th}\label{mainclt}
Let $\pi$ be a uniformly randomly selected $b_2$-regular permutation, $X_i$ be the shifted indicators of $k$-cycles of $\pi$ and $\sigma^2 = \sigma_{n,k}^2$. Further, define $W = \frac{1}{\sigma} \sum_i X_i$. Then for $Z~N(0,1)$, 

\[
d_W(W,Z) \leq \frac{4k^2}{\sigma^3} a_{n,k} + \frac{\sqrt{28} (2k)^{3/2}}{\sqrt{\pi} \sigma^2} \sqrt{b_{n,k}},
\]

where $\sigma^2 = \Theta(n)$, $a_{n,k} = \Theta(n)$ and $b_{n,k} = \Theta(n)$. Therefore the upperbound is of $\Theta(n^{-1/2})$.

\end{Th}

\begin{proof}
Follows from Theorem \ref{localdepclt} and the above calculations.
\end{proof}

\forceindent Our main result, Theorem \ref{r=2clt} is a direct corollary of Theorem \ref{mainclt}. Moreover, the same result can be stated using parts of a random composition, using the bijection.

\begin{Cor}
Let $\lambda$ be a randomly chosen composition of $n$. For $n$ large enough, number of $k$-parts of $\lambda$ has approximate Normal distribution, i.e.,
\[
\# \mbox{of k-parts} \sim N\left(\frac{n-k+3}{2^{k+1}}, \frac{(2^{k+1}-2k+3)n + 3k(k-4) + (3-k)2^{k+1}+ 5}{4^{k+1}}\right).
\]\qed
\end{Cor}

%Baska bise?

\section{Future Work} %generalization of b_2 to b_r

\forceindent One natural next direction for this study is to generalize the results for a wider family one-sided restictions. Specifically, we would like to find central limit theorems similar to Theorem \ref{r=2clt}, but for the more general restriction $\pi(i) \geq i-r+1$, for $r > 2$. The case we studied in this paper would be when $r$ is equal to $2$. The general restriction, $\pi(i) \geq i-r+1$ will produce $b_r$-regular permutations, where $b_r = [1, \ldots, 1, 2, \ldots, n-r+1]$ is the restriction vector with $r$ ones and increasing consecutively. Diaconis et al \cite{DiaconisGrahamHolmes} present an algorithm (Lemma 3.2) to produce a uniformly random $b$-regular permutations. 

\begin{Prop}
Let $b_1 \leq b_2 \leq \ldots \leq b_n$ be positive integers with $b_i \leq i$, $i\in\{1,\ldots,n\}$. The following algorithm results in a uniform choice from $S_b$. Begin with a list containing $1, 2, \ldots, n$.

\begin{itemize}
\item Choose $\pi(n)$ uniformly from $J = \{j: j\geq b_n\}$. Delete $\pi(n)$ from the list of choices.
\item Choose $\pi(n-1)$ uniformly from the elements $j$ in the current list $J$ with $j \geq b_{n-1}$. Delete $\pi(n-1)$ from the list of choices.\\
\vdots
\item Go on till you choose $\pi(1)$.
\end{itemize}\qed
\end{Prop}

\forceindent Using this algorithm, we produced simulations of uniformly random $b_r$-regular permutations and obtained the histograms for number of any fixed size cycles. For large $n$ they seem to follow the Normal distribution as we expected after establishing the result for the $b_2$ case. We present the number of fixed points for $r=3$ as an example, Figure \ref{fig:r3n1000}.

\begin{figure}[h]
  \includegraphics[width=5in]{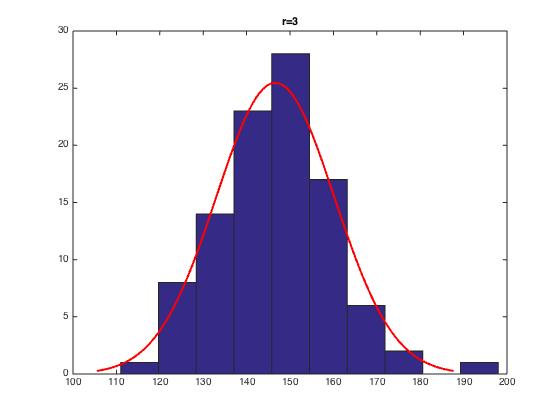}
  \caption{$n = 1000$}
  \label{fig:r3n1000}
\end{figure}

\forceindent Using Proposition 2.2, we can calculate the expected number of fixed points in the $r=3$ case to be $\frac{4n+14}{27}$, which is equal to $148.67$ when $n=1000$. We observe that the histogram is centered around this value with an approximately Normal distribution. Similar simulations encourage us to pursue the proposed central limit theorems. Despite the fact that local dependence structure is more complicated, we believe the following conjecture is true.

\begin{Conj}\label{localdep}

For $b_r$-regular permutations, any two cycle indicators are independent if and only if they are $r-1$-separated, that is, these two indicators have a distance of at least $r-1$ between the numbers their cycles include (recalling that when $r=2$, cycle indicators were $1$-separated).

\end{Conj}

\forceindent Another important step is the calculation of moments. Permanent calculations is viable for small cycle sizes and lower moments, but does not have the generality and flexibility of a cycle-index. We used the bijection between $b_2$-regular permutations and compositions to establish the cycle index. A similar bijection would also help in the construction of the cycle index for general $b_r$-regular permutations. For insance, in the case $r=3$, we observe that the number of $b_3$-regular permutations and the number of compositions of $n-1$ with every part potentially having two different colours are equal. One way to check if a given sequence of positive integers corresponds to the sequence of numbers of a certain type of coloured compositions is using the method presented by Abrate et al \cite{Abrateetal}, which confirms our finding. 

\forceindent For $b_2$-regular permutations, each permutation was mapped to composition that uniquely corresponded to the permutation's cycle structure. For instance, the composition $(4,2,1,3)$ was mapped to a permutation with the cycle structure $4-2-1-3$. This worked out well, since in $b_2$ regular permutations every cycle structure had multiplicity one. For $r > 2$, these multiplicities are potentially much higher than one, which introduces an extra layer of complexity. Unlike what happened with the $b_2$-regular permutations, the coloured compositions are not unique in terms of the cycle structures they are mapped to. We need to find the bijection and map permutations to coloured compositions, and then map these coloured compositions to cycle types in a way that would count their multiplicities.

\forceindent Another interesting observation regarding the bijection between $b_2$-regular permutations and compositions was that it preserved the structure of the random walks over these objects. Considering the random transposition walk over the $b_2$-regular permutations, any two $b_2$-regular permutations that were transposition neighbors, were mapped to two compositions that were neighbors in the walk defined over the set of compositions. In this walk, we either glue two adjacent parts together, or rip apart one part (greater than $1$) into any two possible pieces. For $b_3$-regular permutations we were able to identify a similar random walk that was preserved. In this case the random walk over the colored compositions with each part potentially having two colours is similar to random walk over the regular compositions, with the exception that ``black'' parts act in an ``even nature'' and ``white'' parts act in an ``odd nature'', without loss of generality. For example, a black $1$ and a black $3$ can be composed into a black $4$, but a white $2$ and a black $2$ can be combined in a white $4$. Decomposing works similarly; a black part might be decomposed into two smaller black parts or two smaller white parts, whereas a white part has to be decomposed into smaller parts with different colours. The adjacency graphs in this random walk over two coloured compositions and $b_3$-regular permutations are equivalent. We believe this is the key in constructing the bijection.

\forceindent Lastly, an alternative method of proving central limit theorems is the method of moments. This method is particularly useful when we don't necessarily have a decent grasp on the dependence structure, but have much information on the moments of the subject variable. For instance if were to obtain the cycle index for $b_r$-regular permutations, while not being able to prove Conjecture \ref{localdep}, method of moments would be the most viable.% Proving Theorem \ref{r=2clt} using this method would help identify a potential way to prove similar limiting distribution results for many permutation statistics.

\end{document}